\documentclass[11pt,reqno]{article}
\usepackage{mathrsfs}
\usepackage[colorlinks=true]{hyperref}
\usepackage{amsmath,amssymb,amsthm}
\newtheorem{theorem}{Theorem}[section]
\newtheorem{lemma}[theorem]{Lemma}
\newtheorem{proposition}[theorem]{Proposition}
\newtheorem{corollary}[theorem]{Corollary}

\newtheorem{remark}[theorem]{Remark}
\usepackage[english]{babel}
\numberwithin{equation}{section}

\title{\bf Regular Strichartz estimates in Lorentz-type spaces with application to the $H^s$-critical inhomogeneous biharmonic NLS equation}

\author{{RoeSong Jang, JinMyong An and JinMyong Kim$^*$}\\
\footnotesize{Faculty of Mathematics, {\bf Kim Il Sung} University, Pyongyang, Democratic People's Republic of Korea}\\
\footnotesize{$^*$ Corresponding Author: JinMyong Kim(jm.kim0211@ryongnamsan.edu.kp)}}
\date{}
\begin{document}
\maketitle
\begin{abstract}
In this paper, we investigate the Cauchy problem for the $H^s$-critical inhomogeneous biharmonic nonlinear Schr\"{o}dinger (IBNLS) equation
\[iu_{t}\pm \Delta^{2} u=\lambda |x|^{-b}|u|^{\sigma}u,~u(0)=u_{0} \in H^{s} (\mathbb R^{d}),\]
where $\lambda\in \mathbb C$, $d\ge 3$, $1\le s<\frac{d}{2}$, $0<b<\min \left\{4,2+\frac{d}{2}-s \right\}$ and $\sigma=\frac{8-2b}{d-2s}$.
First, we study the properties of Lorentz-type spaces such as Besov-Lorentz spaces and Triebel-Lizorkin-Lorentz spaces.
We then derive the regular Strichartz estimates for the corresponding linear equation in Lorentz-type spaces.
Using these estimates, we establish the local well-posedness as well as the small data global well-posedness and scattering in $H^s$ for the $H^s$-critical IBNLS equation under less regularity assumption on the nonlinear term than in the recent work \cite{AKR24}. This result also extends the ones of \cite{SP23,SG24} by extending the validity of $d$, $b$ and $s$. Finally, we give the well-posedness result in the homogeneous Sobolev spaces $\dot{H}^s$.
\end{abstract}

\noindent \emph{Mathematics Subject Classification (2020)}. Primary 35Q55, 35A01; Secondary 46E35, 46E30, 42B25.

\noindent \emph{Keywords}. Inhomogeneous biharmonic nonlinear Schr\"{o}dinger equation, critical, well-posedness, regular Strichartz estimates, Lorentz-type spaces, Litttlewood-Paley decomposition.\\

\section{Introduction}\label{sec 1.}

In this paper, we are concerned with the Cauchy problem for the inhomogeneous biharmonic nonlinear Schr\"{o}dinger (IBNLS) equation
\begin{equation} \label{GrindEQ__1_1_}
\left\{\begin{array}{l} {iu_{t} \pm \Delta^{2}u=\lambda|x|^{-b} |u|^{\sigma} u,~(t,x)\in \mathbb R\times \mathbb R^{d},}\\
{u(0,x)=u_{0}(x) \in H^{s}(\mathbb R^{d})}, \end{array}\right.
\end{equation}
where $d\in \mathbb N$, $s\ge0$, $0<b<4$, $\lambda\in \mathbb C$ and $0<\sigma<\infty$.
The case $b=0$ is the well-known classic biharmonic nonlinear Schr\"{o}dinger equation (also called fourth-order NLS equation), which has been widely studied during the last two decades. See, for example, \cite{D18B,D21,D22,K96,KS97,LZ21,MAK24} and the references therein.
Recently, the inhomogeneous biharmonic NLS equation \eqref{GrindEQ__1_1_} with $b>0$ has also attracted a lot of interest.
The local well-posedness as well as the small data global well-posedness and scattering in $H^s$ with $s\ge 0$ were studied by \cite{AKR22,AKR24,ARK23,GP20,LZ212,SP23,SG24}.
Meanwhile, the global existence/scattering and blow-up of $H^2$-solution have also been studied by several authors. See, for example, \cite{BMS23,BS23,CG21, CGP22, GP22, LZ212, S21,S22,SG22}.

The IBNLS equation \eqref{GrindEQ__1_1_} is invariant under scaling $u_{\alpha}(t,x)=\alpha^{\frac{4-b}{\sigma}}u(\alpha^{4}t,\alpha x ),~\alpha >0$.
An easy computation shows that
\begin{equation} \nonumber
\left\|u_{\alpha}(t)\right\|_{\dot{H}^{s}}=\alpha^{s+\frac{4-b}{\sigma}-\frac{d}{2}}\left\|u(t)\right\|_{\dot{H}^{s}}.
\end{equation}
We thus define the critical Sobolev index
\begin{equation} \nonumber
s_{c}:=\frac{d}{2}-\frac{4-b}{\sigma}.
\end{equation}
Putting
\begin{equation} \label{GrindEQ__1_2_}
\sigma_{c}(s):=
\left\{\begin{array}{cl}
{\frac{8-2b}{d-2s},} ~&{{\rm if}~s<\frac{d}{2},}\\
{\infty,}~&{{\rm if}~s\ge \frac{d}{2},}
\end{array}\right.
\end{equation}
we can easily see that $s>s_{c}$ is equivalent to $\sigma<\sigma_{c}(s)$. If $s<\frac{d}{2}$, then $s=s_{c}$ is equivalent to $\sigma=\sigma_{c}(s)$.

For initial data $u_{0}\in H^{s}(\mathbb R^{d})$, we say that the Cauchy problem \eqref{GrindEQ__1_1_} is $H^{s}$-critical (for short, critical) if $0\le s<\frac{d}{2}$ and $\sigma=\sigma_{c}(s)$.
If $s\ge 0$ and $\sigma<\sigma_{c}(s)$, then the problem \eqref{GrindEQ__1_1_} is said to be  $H^{s}$-subcritical (for short, subcritical).

Throughout the paper, $(q,r)$ is said to be biharmonic admissible (for short, $(q,r)\in B$) if
\begin{equation} \label{GrindEQ__1_3_}
\left\{\begin{array}{ll}
{2\le r\le\frac{2d}{d-4}},~&{{\rm if}~d\ge 5,}\\
{2\le r<\infty,}~&{{\rm if}~d\le 4,}
\end{array}\right.
\end{equation}
and
\begin{equation} \label{GrindEQ__1_4_}
\frac{4}{q} =\frac{d}{2} -\frac{d}{r} .
\end{equation}

In addition, $(q,r)\in B_0$ means that $(q,r)\in B$ with $\max\{\frac{d-4}{2d},0\}<\frac{1}{r}<\frac{1}{2}$.

A pair $(q,r)$ is said to be Schr\"{o}dinger admissible (for short, $(q,r)\in S$) if
\begin{equation} \label{GrindEQ__1_5_}
\left\{\begin{array}{ll}
{2\le r\le\frac{2d}{d-2}},~&{{\rm if}~d\ge 3,}\\
{2\le r<\infty,}~&{{\rm if}~d\le 2,}
\end{array}\right.
\end{equation}
and
\begin{equation} \label{GrindEQ__1_6_}
\frac{2}{q} =\frac{d}{2} -\frac{d}{r} .
\end{equation}

In this paper, we focus on the $H^s$-critical case, i.e. $\sigma=\frac{8-2b}{d-2s}$ with $0<s<\frac{d}{2}$.
This case was already studied by \cite{AKR24,SP23,SG24}. They studied the local well-posedness as well as the small data global well-posedness in $H^s$ by using different methods from each other. Saanouni-Peng \cite{SP23} used the Strichartz estimates in some weighted Lebesgue spaces adapted to handle the inhomogeneous term $|x|^{-b}$. And Saanouni-Ghanmi \cite{SG24} used the fractional Hardy type inequality, following the idea of \cite{AK23}.
Recently, the authors in \cite{AKR24} give the unified proof for both of $H^s$-subcritical case and $H^s$-subcritical case by using the Sobolev-Lorentz spaces theory .
More precisely, they proved that the Cauchy problem \eqref{GrindEQ__1_1_} with $d\in \mathbb N$, $0\le s<\min\left\{2+\frac{d}{2},d\right\}$ , $0<b<\min \left\{4,\; d-s,\; 2+\frac{d}{2}-s \right\}$ and $0<\sigma\le \sigma_{c}(s)$ is locally well-posed in $H^s$ if either $\sigma$ is an even integer or
\footnote[1]{~For $s\in \mathbb R$, $\left\lceil s\right\rceil$ denotes the minimal integer which is larger than or equal to $s$}
\begin{equation}\label{GrindEQ__1_7_}
\sigma>\left\lceil s\right\rceil-1.
\end{equation}

The condition \eqref{GrindEQ__1_7_} was assumed to ensure the regularity of the nonlinear term $|u|^{\sigma} u$ when $\sigma$ is not an even integer.
Obviously, this condition becomes trivial when $s\le 1$. However, when $s>1$, this condition is not so good as in $H^s$-subcritical case obtained by \cite{AKR22,ARK23,LZ212} and lead to restrictive assumption on $d$ and $b$.
In particular, when we study the $H^2$-theory for \eqref{GrindEQ__1_1_}, this condition leads to too restrictive assumption: $d\le 11$ and $b<6-\frac{d}{2}$ .

The main purpose of this paper is to improve the condition \eqref{GrindEQ__1_7_} in the $H^s$-critical case with $s>1$.

As mentioned above, in the $H^s$-subcritical case $\sigma<\sigma_{\rm c}(s)$, the condition \eqref{GrindEQ__1_7_} was already improved by \cite{AKR22,ARK23,LZ212}. However, the methods of \cite{AKR22,ARK23,LZ212} cannot be applied in the $H^s$-critical case.
To overcome this difficulty, we apply the theory of Besov-Lorentz spaces and Triebel-Lizorkin-Lorentz spaces to derive the regular Strichartz estimates for the biharmonic Schr\"{o}dinger semi-group $S(t)=e^{\pm it\Delta^{2} }$ in the framework of Lorentz-type spaces.
Using this regular Strichartz estimates and the contraction mapping principle, we have the following improved result about the local well-posedness as the small data global well-posedness and scattering in $H^s$ for the critical IBNLS equation \eqref{GrindEQ__1_1_}.

\begin{theorem}[Well-posedness in $H^s$]\label{thm 1.1.}
Let $d\ge 3$, $1\le s<\frac{d}{2}$ , $0<b<\min \left\{4,2+\frac{d}{2}-s \right\}$ and $\sigma=\frac{8-2b}{d-2s}$. If $\sigma$ is not an even integer, assume further
\begin{equation} \label{GrindEQ__1_8_}
\sigma>\lceil s\rceil-2.
\end{equation}
Then for any $u_{0}\in H^{s}(\mathbb R^{d}) $, there exist $T_{\max }=T_{\max }(u_{0})>0$ and $T_{\min }=T_{\min }(u_{0})>0$ such that \eqref{GrindEQ__1_1_} has a unique, maximal solution satisfying
\begin{equation} \nonumber
u\in C(\left(-T_{\min } ,T_{\max } \right),H^{s} )\cap L^{q}_{\rm loc}((-T_{\min } ,T_{\max } ),H_{r,2}^{s} ).
\end{equation}
for any  $(q,r)\in B$.
If $\left\| u_{0} \right\| _{\dot{H}^{s}} $ is sufficiently small, then the above solution is global and scatters.
\end{theorem}
\begin{remark}\label{rem 1.2.}
\textnormal{
The authors in \cite{AKR24} established the well-posedness under the assumption that $0<b<\min \left\{4, d-s, 2+\frac{d}{2}-s \right\}$ and \eqref{GrindEQ__1_7_} are satisfied. Obviously, Theorem \ref{thm 1.1.} improves these conditions.
}\end{remark}
As an immediate consequence of Theorem \ref{thm 1.1.}, we have the following well-posedness result in the energy space $H^2$.
\begin{corollary}[Well-posedness in $H^2$]\label{cor 1.3.}
Let $d\ge 5$, $0<b<\min \left\{4,\frac{d}{2}\right\}$ and $\sigma=\frac{8-2b}{d-4}$. Then for any $u_{0}\in H^{2}(\mathbb R^{d}) $, there exist $T_{\max }=T_{\max }(u_{0})>0$ and $T_{\min }=T_{\min }(u_{0})>0$ such that \eqref{GrindEQ__1_1_} has a unique, maximal solution satisfying
\begin{equation} \nonumber
u\in C(\left(-T_{\min } ,T_{\max } \right),H^{2} )\cap L^{q}_{\rm loc}((-T_{\min } ,T_{\max } ),H_{r,2}^{2} ).
\end{equation}
for any $(q,r)\in B$.
Moreover, if $\left\|u_{0}\right\|_{\dot{H}^{2}}$ is sufficiently small, then the above solution is global and scatters.
\end{corollary}

\begin{remark}\label{rem 1.4.}
\textnormal{\begin{itemize}
              \item In Corollary \ref{cor 1.3.}, the restriction $d\ge 5$ is natural, since Corollary \ref{cor 1.3.} only covers the $H^2$-critical case. Furthermore, the assumption \eqref{GrindEQ__1_8_} becomes trivial, due to the fact that $s=2$.
              \item Corollary \ref{cor 1.3.} also improves the results of \cite{SP23,SG24} by extending the validity of $d$ and $b$. In fact, Saanouni-Ghanmi \cite{SG24} assumed $5\le d\le 11$ and $0<b<\min \left\{4,\frac{d}{2},\frac{12+d}{d}\right\}$. And Saanouni-Peng \cite{SP23} assumed $0<b<\min \left\{4,\frac{d}{2},\frac{d(4+d)}{4(d-2)}\right\}$.
            \end{itemize}
}\end{remark}

We also have the following well-posedness result in the homogeneous Sobolev spaces $\dot{H}^s$.
\begin{theorem}[Well-posedness in $\dot{H}^s$]\label{thm 1.5.}
Let $d\ge 3$, $1\le s<\frac{d}{2}$ , $0<b<\min \left\{4,2+\frac{d}{2}-s \right\}$ and $\sigma=\frac{8-2b}{d-2s}$. If $\sigma$ is not an even integer, assume further
\begin{equation} \label{GrindEQ__1_9_}
\sigma>\lceil s\rceil-1.
\end{equation}
Then for any $u_{0}\in \dot{H}^{s}(\mathbb R^{d}) $, there exist $T_{\max }=T_{\max }(u_{0})>0$ and $T_{\min }=T_{\min }(u_{0})>0$ such that \eqref{GrindEQ__1_1_} has a unique, maximal solution satisfying
\begin{equation} \nonumber
u\in C(\left(-T_{\min } ,T_{\max } \right),\dot{H}^{s} )\cap L^{q}_{\rm loc}((-T_{\min } ,T_{\max } ),\dot{H}_{r,2}^{s} ).
\end{equation}
for any  $(q,r)\in B$.
If $\left\| u_{0} \right\| _{\dot{H}^{s}} $ is sufficiently small, then the above solution is global and scatters.
\end{theorem}
\begin{remark}\label{rem 1.6.}
\textnormal{One can easily verify that the result of Theorem \ref{thm 1.5.} still holds for the $H^s$-subcritical case $\sigma<\frac{8-2b}{d-2s}$ by modifying the proof in Section \ref{sec 5.} slightly.
}\end{remark}

This paper is organized as follows. In Section \ref{sec 2.}, we give some preliminary results related to our problem. In Section \ref{sec 3.}, we study the properties of Lorentz-type spaces and derive the regular Strichartz estimates for $S(t)=e^{\pm it\Delta^{2} }$ in the framework of Lorentz-type spaces. In Section \ref{sec 4.}, we prove Theorem \ref{thm 1.1.}. Theorem \ref{thm 1.5.} is proved in Section \ref{sec 5.}.

\section{Preliminaries}\label{sec 2.}

In this section, we give some preliminary results related to our problem.

First of all, we introduce some notation used in this paper.
Throughout the paper, $\mathscr{F}$ denotes the Fourier transform, and the inverse Fourier transform is denoted by $\mathscr{F}^{-1}$.
$C>0$ stands for a positive universal constant, which may be different at different places. $a\lesssim b$ means $a\le Cb$ for some constant $C>0$. $a\sim b$ expresses $a\lesssim b$ and $b\lesssim a$.
Given normed spaces $X$ and $Y$, $X\hookrightarrow Y$ means that $X$ is continuously embedded in $Y$.
For $p\in \left[1,\;\infty \right]$, $p'$ denotes the dual number of $p$, i.e. $1/p+1/p'=1$.

Let us recall the definitions and the properties of some function spaces.

As in \cite{BL76,T78}, for $s\in \mathbb R$, $1\le q\le\infty$ and normed space $X$, the space $\dot{l}_q^s(X)$ denotes the space of all sequences $a=\{a_k\}_{k\in \mathbb Z}\subset X$ such that
$$
\left\|a\right\|_{\dot{l}_q^s(X)}=\left(\sum_{k\in \mathbb Z}\left[2^{-ks}\left\|a_k\right\|_X\right]^q\right)^{\frac{1}{q}}<\infty,
$$
with usual modification when $q=\infty$.
In the case $X=\mathbb C$, we write $\dot{l}_q^s$ instead of $\dot{l}_q^s(\mathbb C)$.
Similarly, the space
the space $l_q^s(X)$ denotes the space of all sequences $a=\{a_k\}_{k=0}^{\infty}\subset X$ such that
$$
\left\|a\right\|_{l_q^s(X)}=\left(\sum_{k=0}^{\infty}\left[2^{-ks}\left\|a_k\right\|_X\right]^q\right)^{\frac{1}{q}}<\infty,
$$
with usual modification when $q=\infty$.

As in \cite{WHHG11}, for $s\in \mathbb R$ and $1<p<\infty $, we denote by $H_{p}^{s} (\mathbb R^{d} )$ and $\dot{H}_{p}^{s} (\mathbb R^{d} )$ the nonhomogeneous Sobolev space and homogeneous Sobolev space, respectively. The norms of these spaces are given as
$$
\left\| f\right\| _{H_{p}^{s} (\mathbb R^{d} )} =\left\| (I-\Delta)^{s/2} f\right\| _{L^{p} (\mathbb R^{d} )} , \;\left\| f\right\| _{\dot{H}_{p}^{s} (\mathbb R^{d} )} =\left\| (-\Delta)^{s/2} f\right\| _{L^{p} (\mathbb R^{d} )},
$$
where $(I-\Delta)^{s/2}f =\mathscr{F}^{-1} \left(1+|\xi|^{2} \right)^{s/2} \mathscr{F}f$ and $(-\Delta)^{s/2}f =\mathscr{F}^{-1} |\xi|^{s} \mathscr{F}f$. As usual, we abbreviate $H_{2}^{s} (\mathbb R^{d} )$ and $\dot{H}_{2}^{s} (\mathbb R^{d} )$ as $H^{s} (\mathbb R^{d} )$ and $\dot{H}^{s} (\mathbb R^{d} )$, respectively.

For $0<p,\; q\le \infty $, we denote by $L^{p,q} (\mathbb R^{d})$ the Lorentz space (cf. \cite{G14}).
The quasi-norms of these spaces are given by
$$
\left\|f\right\|_{L^{p,q} (\mathbb R^{d})}=\left(\int_{0}^{\infty}{\left(t^{\frac{1}{p}}f^{*}(t)\right)^{q}
\frac{dt}{t}}\right)^{\frac{1}{q}},~~\textnormal{when}~~0<q<\infty,
$$
$$
\left\|f\right\|_{L^{p,\infty} (\mathbb R^{d})}=\sup_{t>0}t^{\frac{1}{p}}f^{*}(t),~~\textnormal{when}~~q=\infty,
$$
where $f^{*}(t)=\inf\left\{\tau:M^{d}\left(\left\{x:|f(x)|>\tau\right\}\right)\le t\right\}$, with $M^{d}$ being the Lebesgue measure in $\mathbb R^{d}$. Note that $L^{p,q} \left(\mathbb R^{d}
\right)$ is a quasi-Banach space for $0<p,q\le \infty $.  When $1<p<\infty$ and $1\le q \le \infty$,  $L^{p,q} \left(\mathbb R^{d}
\right)$ can be turned into a Banach space via an equivalent norm.  In particular $L^{p,p} (\mathbb R^{d})=L^{p}
(\mathbb R^{d})$, while $L^{p,\infty}$ corresponds to weak $L^p$ space.
These spaces are natural in the context of \eqref{GrindEQ__1_1_} due to the fact that $|x|^{-b}\in L^{\frac{d}{b},\infty}(\mathbb R^d)$.
In general, we have the embedding $L^{p,q}\hookrightarrow L^{p,r}$ for $q<r$.
Note also that $\left\| \left|f\right|^{r} \right\| _{L^{p,q} } =\left\| f\right\| _{L^{pr,qr}}^{r}$ for $1\le p,\;r<\infty $, $1\le q\le \infty $.
We also have the following H\"{o}lder's inequality in Lorentz spaces.
\begin{lemma}[H\"{o}lder's inequality in Lorentz spaces, \cite{O63}]\label{lem 2.1.}
Let $1<p,p_{1},p_{2}<\infty $, $1\le q,q_{1} ,q_{2}\le \infty $ and
$$
\frac{1}{p}=\frac{1}{p_{1}} +\frac{1}{p_{2}} ,~\frac{1}{q}=\frac{1}{q_{1} } +\frac{1}{q_{2} }.
$$
The we have
$$
\left\| fg\right\| _{L^{p,q} } \lesssim \left\| f\right\|
_{L^{p_{1},q_{1} } } \left\| g\right\| _{L^{p_{2},q_{2} } } .
$$
\end{lemma}

As in  \cite{AT21,AK232,AKR24,HYZ12}, for $s\in \mathbb R$, $1<p<\infty$ and $1\le q \le \infty$, the nonhomogeneous Sobolev-Lorentz space ${H}^{s}_{p,q}(\mathbb R^{d})$ is defined as the set of tempered distribution $f\in \mathscr{S}'(\mathbb R^d)$ such that $(I-\Delta)^{s/2}f \in L^{p,q}(\mathbb R^{d})$, equipped with the norm
$$\left\|f\right\|_{{H}^{s}_{p,q}(\mathbb R^{d})}=\left\|(I-\Delta)^{s/2}f \right\|_{L^{p,q}(\mathbb R^{d})}.$$
The homogeneous Sobolev-Lorentz space $\dot{H}^{s}_{p,q}(\mathbb R^{d})$ is defined as the set of equivalence classes of distribution $f\in \mathscr{S}'(\mathbb R^d)/{\mathscr{P}(\mathbb R^d)}$ such that $(-\Delta)^{s/2}f \in L^{p,q}(\mathbb R^{d})$, equipped with the norm
$$\left\|f\right\|_{\dot{H}^{s}_{p,q}(\mathbb R^{d})}=\left\|(-\Delta)^{s/2}f\right\|_{L^{p,q}(\mathbb R^{d})},$$
where $\mathscr{P}(\mathbb R^d)$ denotes the set of polynomials with $d$ variables.
\begin{lemma}[\cite{AK232}]\label{lem 2.2.}
Let $s\ge 0$, $1<p<\infty$ and $1\le q_{1}\le q_{2}\le \infty$. Then we have

$(a)$ $\dot{H}^{s}_{p,1}\hookrightarrow \dot{H}^{s}_{p,q_{1}} \hookrightarrow\dot{H}^{s}_{p,q_{2}} \hookrightarrow \dot{H}^{s}_{p,\infty}$,

$(b)$ $\dot{H}^{s}_{p,p}=\dot{H}^{s}_{p}$.
\end{lemma}
\begin{lemma}[\cite{AKR24}]\label{lem 2.3.}
Let $s\ge 0$, $1<p<\infty $ and $1\le q\le \infty$. Then we have $H^{s}_{p,q}=L^{p,q}\cap \dot{H}^{s}_{p,q}$ with
\[\left\|f\right\|_{H^{s}_{p,q}}\sim \left\|f\right\|_{L^{p,q}}+\left\|f\right\|_{\dot{H}^{s}_{p,q}}.\]
\end{lemma}

\begin{lemma}[\cite{AKR24}]\label{lem 2.4.}
Let $-\infty < s_{2} \le s_{1} <\infty $ and $1<p_{1} \le p_{2} <\infty $ with $s_{1} -\frac{d}{p_{1} } =s_{2} -\frac{d}{p_{2} } $. Then for any $1\le q\le \infty$, there holds the embeddings:
$$
\dot{H}_{p_{1},q}^{s_{1} } \hookrightarrow \dot{H}_{p_{2},q}^{s_{2}},~H_{p_{1},q}^{s_{1} } \hookrightarrow H_{p_{2},q}^{s_{2}}.
$$
\end{lemma}
\begin{lemma}[Fractional product rule under Lorentz norms, \cite{CN16}]\label{lem 2.5.}
Let $s\ge 0$, $1<p,p_{1},p_{2},p_{3},p_{4} <\infty$ and $1\le q,q_{1},q_{2},q_{3},q_{4} \le \infty$. Assume that
\begin{equation}\nonumber
\frac{1}{p} =\frac{1}{p_{1} }+\frac{1}{p_{2}}=\frac{1}{p_{3} }+\frac{1}{p_{4}},~\frac{1}{q} =\frac{1}{q_{1} }+\frac{1}{q_{2}}=\frac{1}{q_{3} }+\frac{1}{q_{4}}.
\end{equation}
Then we have
\begin{equation}\nonumber
\left\|fg\right\| _{\dot{H}_{p,q}^{s} } \lesssim \left\|f\right\| _{\dot{H}_{p_1,q_1}^{s} } \left\| g\right\| _{L^{p_{2},q_{2}}} +\left\| f\right\| _{L^{p_{3},q_{3}}}
\left\| g\right\| _{\dot{H}_{p_4,q_4}^{s} }.
\end{equation}
\end{lemma}
\begin{lemma}[Fractional chain rule under Lorentz norms, \cite{AT21,AKR24}]\label{lem 2.6.}
Let $s>0$ and $F\in C^{\left\lceil s\right\rceil } \left(\mathbb C\to \mathbb C\right)$. Then for $1<p,p_{1_k},p_{2_k},p_{3_k}<\infty$ and $1\le q,q_{1_k},q_{2_k},q_{3_k}<\infty$ satisfying
\begin{equation} \nonumber
\frac{1}{p}=\frac{1}{p_{1_k}}+\frac{1}{p_{2_k}}+\frac{k-1}{p_{3_k}},
~\frac{1}{q}=\frac{1}{q_{1_k}}+\frac{1}{q_{2_k}}+\frac{k-1}{q_{3_k}},~k=1,2,\ldots, \lceil s\rceil,
\end{equation}
we have
\begin{equation} \nonumber
\left\|(-\Delta)^{s/2} F(u)\right\| _{L^{p,q}} \lesssim \sum_{k=1}^{\lceil s\rceil}\left\| F^{(k)}(u)\right\| _{L^{p_{1_k},q_{1_k}}} \left\|(-\Delta)^{s/2}u\right\|_{L^{p_{2_k},q_{2_k}}}\left\|u\right\|_{L^{p_{3_k},q_{3_k}}}^{k-1},
\end{equation}
where the $k$-th order derivative of $F(z)$ is defined under the identification $\mathbb C=\mathbb R^2$ $($see e.g. $\cite{AK23,AK232,AKR24}.)$
\end{lemma}
As an immediate consequence of Lemma \ref{lem 2.6.}, we have the following useful nonlinear estimates in Sobolev-Lorentz spaces $\dot{H}_{p,q}^{s}$ with $s\ge 0$.
\begin{corollary}\label{cor 2.7.}
Let $F(u)=|u|^{\sigma}u$ or $F(u)=|u|^{\sigma+1}$. Assume $s\ge 0$ and $\sigma >0$. If $\sigma$ is not an even integer, assume further $\sigma >\lceil s\rceil -1$. Then for $1<p,p_{1},p_{2}<\infty$ and $1\le q,q_{1},q_{2}<\infty$ satisfying
\begin{equation} \nonumber
\frac{1}{p}=\frac{\sigma}{p_{1}}+\frac{1}{p_{2}},~\frac{1}{q}=\frac{\sigma}{q_{1}}+\frac{1}{q_{2}},
\end{equation}
we have
\begin{equation} \nonumber
\left\|F(u)\right\| _{\dot{H}_{p,q}^{s}} \lesssim \left\|u\right\| _{L^{p_{1},q_{1}}}^{\sigma} \left\|u\right\|_{\dot{H}_{p_{2},q_{2}}^{s}}.
\end{equation}
\end{corollary}
We also recall the useful fact concerning to the term $|x|^{-b}$ with $b>0$.
\begin{lemma}\label{lem 2.8.}
Let $b>0$, $s\ge 0$ and $b+s<d$. Then we have $|x|^{-b}\in \dot{H}_{d/(b+s),\infty}^s(\mathbb R^d)$.
\end{lemma}
\begin{proof}
The result follows from the fact that $|x|^{-b}\in L^{\frac{d}{b},\infty}(\mathbb R^d)$ and $(-\Delta)^{s/2}(|x|^{-b})=C_{d,b}|x|^{-b-s}$.
\end{proof}

Finally, we recall the standard Strichartz estimates in Sobolev-Lorentz spaces.

\begin{lemma}[\cite{AKR24,KT98}]\label{lem 2.9.}
Let $S(t)=e^{\pm it\Delta^{2} }$ and $s\in \mathbb R$. Then for any $(q,r)\in B$ and $(\bar{q},\bar{r})\in B$, we have
\begin{equation} \label{GrindEQ__2_1_}
\left\| S(t)\phi \right\| _{L^{q}(\mathbb R,\dot{H}_{r,2}^s)} \lesssim\left\| \phi \right\| _{\dot{H}^{s} } ,
\end{equation}
\begin{equation} \label{GrindEQ__2_2_}
\left\|\int_{0}^{t}S(t-\tau)f(\tau)d\tau\right\|_{L^{q}(\mathbb R,\dot{H}_{r,2}^s)}\lesssim\left\|f\right\|_{L^{\bar{q}'}(\mathbb R,\dot{H}_{\bar{r}',2}^s)}.
\end{equation}
\end{lemma}

\section{Regular Strichartz estimates in Lorentz-type spaces}\label{sec 3.}
In this section, we study the properties of Lorentz-type spaces and derive the regular Strichartz estimates in Besov-Lorentz and Sobolev-Lorentz spaces.
\subsection{Besov-Lorentz and Triebel-Lizorkin-Lorentz spaces}
In this subsection, we recall the definition of Besov-Lorentz and Triebel-Lizorkin-Lorentz spaces, and study the properties of these spaces.
The Besov-Lorentz and Triebel-Lizorkin-Lorentz spaces were initially introduced in Sec 2.4 of Triebel \cite{T78}.
Later, these spaces have been studied by several authors and applied to the study of partial differential equations such as Navier-Stokes equations and Euler equations. See, for example, \cite{HS19,L23,ST19,XY11,XY12,YCP05,ZWT11}.

We first recall the definition of homogeneous Littlewood-Paley decomposition (cf. \cite{WHHG11}).
Let $\varphi:\mathbb R^d\to \mathbb [0,1]$ be a smooth radial cut-off function such that $\varphi(\xi)=\frac{1}{2}$ for $|\xi|\le 1$ and $\varphi(\xi)=0$ for $|\xi|\ge 2$.
For $k\in \mathbb Z$, let $\varphi_k(\xi)=\varphi(2^{-k}\xi)-\varphi(2^{1-k}\xi)$.
We define the homogeneous Littlewood-Paley decomposition operator $\{\dot{\Delta}_k\}_{k\in \mathbb Z}$ as follows:
$$\dot{\Delta}_k:=\mathscr{F}^{-1}\varphi_k\mathscr{F},~k\in \mathbb Z.$$

Combining the above homogeneous Leittlewood-Paley decomposition operator $\{\dot{\Delta}_k\}_{k\in \mathbb Z}$ with the function spaces $\dot{l}_r^s(L^{p,q})$ and $L^{p,q}(\dot{l}_r^s)$, we can introduce homogeneous Besov-Lorentz spaces $\dot{B}_{(p,q),r}^s$ and Triebel-Lizorkin-Lorentz spaces $\dot{F}_{(p,q),r}^s$. See \cite{L23,XY11,XY12,YCP05,ZWT11} for example.

Assume that $s\in \mathbb R$, $1<p<\infty$ and $1\le q,r\le \infty$. The homogeneous Besov-Lorentz spaces $\dot{B}_{(p,q),r}^s$ and Triebel-Lizorkin-Lorentz spaces $\dot{F}_{(p,q),r}^s$ are define, respectively, to be the sets of equivalence classes of distribution $f\in \mathscr{S}'(\mathbb R^d)/{\mathscr{P}(\mathbb R^d)}$ such that
$$
\left\|f\right\|_{\dot{B}_{(p,q),r}^s}:=\left\|\{\dot{\Delta}_k f\}_{k\in \mathbb Z}\right\|_{\dot{l}_r^s(L^{p,q})}<\infty
$$
and
$$
\left\|f\right\|_{\dot{F}_{(p,q),r}^s}:=\left\|\{\dot{\Delta}_k f\}_{k\in \mathbb Z}\right\|_{L^{p,q}(\dot{l}_r^s)}<\infty.
$$

Replacing $\dot{l}_q^s$-norms and the homogeneous Litttlewood-Paley decomposition operator $\{\dot{\Delta}_k\}_{k\in \mathbb Z}$ by $l_q^s$-norms and the nonhomogeneous Litttlewood-Paley decomposition operator $\{\Delta_k\}_{k=0}^{\infty}$ (cf. \cite{T83,WHHG11}), we can define the nonhomogeneous Besov-Lorentz spaces $B_{(p,q),r}^s$ and Triebel-Lizorkin-Lorentz spaces $F_{(p,q),r}^s$. See e.g. \cite{HS19,ST19} for the properties of these spaces.

In this paper, we only focus on the homogeneous Besov-Lorentz spaces and Triebel-Lizorkin-Lorentz spaces.
If we replace the Lorentz-norms by the corresponding Lebesgue-norms, we get the well-known Besov spaces ${\dot{B}_{p,r}^s}$ and Triebel-Lizorkin spaces ${\dot{F}_{p,r}^s}$. More precisely, we have ${\dot{B}_{(p,p),r}^s}={\dot{B}_{p,r}^s}$ and ${\dot{F}_{(p,p),r}^s}={\dot{F}_{p,r}^s}$.

We recall the interpolation of Triebel-Lizorkin-Lorentz spaces.
\begin{lemma}[\cite{YCP05}]\label{intF}
Let $1<p_0,p_1<\infty$, $1\le q,r,r_0,r_1\le \infty$, $0<\theta<1$ and $s\in\mathbb R$. If $\frac{1}{p}=\frac{1-\theta}{p_0}+\frac{\theta}{p_1}$, then we have the following real interpolation (see. e.g. \cite{BL76,T78} for the definition and properties of real interpolation):
$$
(\dot{F}_{(p_0,r_0),q}^s,\dot{F}_{(p_1,r_1),q}^s)_{\theta,r}=\dot{F}_{(p,r),q}^s.
$$
In particular, we have
$$
(\dot{F}_{p_0,q}^s,\dot{F}_{p_1,q}^s)_{\theta,r}=\dot{F}_{(p,r),q}^s.
$$
\end{lemma}

We also have the following interpolation of Besov-Lorentz spaces.

\begin{lemma}\label{intB}
Let $1<p_0,p_1<\infty$, $1\le q,q_0,q_1,r_0,r_1<\infty$, $s,s_,s_1\in\mathbb R$ and $0<\theta<1$. If
$$
s=(1-\theta)s_0+\theta s_1,~\frac{1}{p}=\frac{1-\theta}{p_0}+\frac{\theta}{p_1}~\textrm{and}~ \frac{1}{q}=\frac{1-\theta}{q_0}+\frac{\theta}{q_1},
$$
then we have the following real interpolation:
$$
(\dot{B}_{(p_0,r_0),q_0}^{s_0},\dot{B}_{(p_1,r_1),q_1}^{s_1})_{\theta,q}=\dot{B}_{(p,q),q}^s.
$$
In particular, we have
$$
(\dot{B}_{p_0,q_0}^{s_0},\dot{B}_{p_1,q_1}^{s_1})_{\theta,q}=\dot{B}_{(p,q),q}^s.
$$
\end{lemma}
\begin{proof}
\cite[Theorem 5.6.2]{BL76} shows that
\footnote[2]{~It is worth mentioning that the authors in \cite{BL76} proved the real interpolation $(l_{q_0}^{s_0}(A_0),l_{q_1}^{s_1}(A_1))_{\theta,q}=l_{q_1}^{s_1}((A_0,A_1)_{\theta, q})$. But one can easily see that the proof and so the result still hold for the corresponding dotted spaces.}
\begin{equation} \label{GrindEQ__3_1_}
(\dot{l}_{q_0}^{s_0}(A_0),\dot{l}_{q_1}^{s_1}(A_1))_{\theta,q}=\dot{l}_{q_1}^{s_1}((A_0,A_1)_{\theta, q}).
\end{equation}
And Theorem 2 in Subsection 1.18.6 of \cite{T78} shows that
\begin{equation} \label{GrindEQ__3_2_}
(L^{p_0,r_0},L^{p_1,r_1})_{\theta,q}=L^{p,q}.
\end{equation}
The results follow directly from \eqref{GrindEQ__3_1_} and \eqref{GrindEQ__3_2_}.
\end{proof}

Using Lemma \ref{intF}, we also have the following interesting result.
\begin{lemma}\label{isoF}
Let $1<p<\infty$, $1\le q,r\le \infty$ and $s, \gamma \in\mathbb R$. Then
we have
$$
\left\|(-\Delta)^{\gamma /2}f\right\|_{\dot{F}_{(p,r),q}^s}\sim \left\|f\right\|_{\dot{F}_{(p,r),q}^{s+\gamma}}.
$$
\end{lemma}
\begin{proof}
Theorem 1 in Subsection 5.2.3 of \cite{T83} shows that
\begin{equation} \label{GrindEQ__3_3_}
\left\|(-\Delta)^{\gamma /2}f\right\|_{\dot{F}_{p,q}^s}\sim \left\|f\right\|_{\dot{F}_{p,q}^{s+\gamma}},
\end{equation}
for any $1<p<\infty$, $1\le q\le\infty$, and $s,\gamma\in \mathbb R$.
\eqref{GrindEQ__3_3_} and Lemma \ref{intF} imply that
\begin{equation} \label{GrindEQ__3_4_}
\left\|(-\Delta)^{\gamma /2}f\right\|_{\dot{F}_{(p,r),q}^s}\lesssim \left\|f\right\|_{\dot{F}_{(p,r),q}^{s+\gamma}},
\end{equation}
for any $1<p<\infty$, $1\le q,r\le\infty$, and $s,\gamma\in \mathbb R$.
\eqref{GrindEQ__3_4_} shows that
\begin{equation} \label{GrindEQ__3_5_}
\left\|(-\Delta)^{-\gamma /2}f\right\|_{\dot{F}_{(p,r),q}^{s+\gamma}}\lesssim \left\|f\right\|_{\dot{F}_{(p,r),q}^{s}}.
\end{equation}
Taking $f=(-\Delta)^{\gamma /2}g$ in \eqref{GrindEQ__3_5_}, we immediately get
\begin{equation} \label{GrindEQ__3_6_}
\left\|g\right\|_{\dot{F}_{(p,r),q}^{s+\gamma}}\lesssim \left\|(-\Delta)^{\gamma /2}g\right\|_{\dot{F}_{(p,r),q}^{s}}.
\end{equation}
Using \eqref{GrindEQ__3_4_} and \eqref{GrindEQ__3_6_}, we immediately get the desired result.
\end{proof}

It is well-known that $\left\|f\right\|_{\dot{H}_{p}^{s}}\sim \left\|f\right\|_{\dot{F}_{p,2}^{s}}$ for $1<p<\infty$ and $s\in\mathbb R$.
See, for example, \cite[Subsection 5.2.3]{T83}.
This result can be extended as follows.
\begin{corollary}\label{equHF}
Let $1<p<\infty$, $1\le q<\infty$ and $s\in\mathbb R$. Then we have
$\left\|f\right\|_{\dot{H}_{p,q}^{s}}\sim \left\|f\right\|_{\dot{F}_{(p,q),2}^{s}}$.
\end{corollary}
\begin{proof}
The case $s=0$ was proved in Theorem 5 of \cite{XY11}, i.e. we have $\left\|f\right\|_{L^{p,q}}\sim \left\|f\right\|_{\dot{F}_{(p,q),2}^{0 }}$.
Using this fact and Lemma \ref{isoF}, we immediately get the desired result.
\end{proof}

We also have the following result for Besov-Lorentz spaces which is similar to Lemma \ref{isoF}.
\begin{lemma}\label{isoB}
Let $1<p<\infty$, $1\le q,r\le \infty$ and $s, \gamma \in\mathbb R$. Then
we have
$$
\left\|(-\Delta)^{\gamma /2}f\right\|_{\dot{B}_{(p,r),q}^s}\sim \left\|f\right\|_{\dot{B}_{(p,r),q}^{s+\gamma}}.
$$
\end{lemma}
\begin{proof}
Noticing that
$$
\dot{\Delta}_k f=\sum_{l=-1}^{l=1}{\dot{\Delta}_{k+l}(\dot{\Delta}_k f)},
$$
it suffices to prove that
$$
\left\|\mathscr{F}^{-1}|\xi|^{\gamma}\varphi_{k+l}\mathscr{F}f\right\|_{L^{p,q}}\lesssim 2^{\gamma k} \left\|f\right\|_{L^{p,q}}~(l=0,\pm 1),
$$
for any $1<p<\infty$, $1\le q\le \infty$, $\gamma\in \mathbb R$ and $k\in \mathbb Z$. And this result follows from the fact that
$$
\left\||\xi|^{\gamma}\varphi_{k+l}\right\|_{M_{p}}\lesssim 2^{\gamma k},
$$
and the real interpolation, where $M_p$ denotes the set of all multiplier on $L^p$. See (5) in the proof of Lemma 6.2.1 on page 140 of \cite{BL76}.
This completes the proof.
\end{proof}
Let us consider the embeddings between the homogeneous Besov-Lorentz spaces, Triebel-Lizorkin-Lorentz spaces and Sobolev-Lorentz spaces.
To this end we recall the following embedding result
\footnote[3]{~In fact, the authors in \cite{ST19} obtained the embedding result for the nonhomogeneous version. But one can see that the proof and so the result still hold for the homogeneous version.}
.
\begin{proposition}[\cite{ST19}]\label{embl}
Let $1<p<\infty$, $1\le q_0,q_1,r_0,r_1<\infty$, $r_0\le r_1$.
\begin{enumerate}
  \item There holds the embedding $\dot{l}_{q_0}^{0}(L^{p,r_0})\hookrightarrow L^{p,r_1}(\dot{l}_{q_1}^{0})$, if $q_0\le \min\{p,q_1,r_1\}$ and one of the following conditions are satisfied:

     $(i)~p\neq q_1,~~(ii)~~p=q_1\ge r_0,~~(iii)~~q_0<p=q_1<r_0$.
  \item There holds the embedding $L^{p,r_0}(\dot{l}_{q_0}^{0})\hookrightarrow \dot{l}_{q_1}^{0}(L^{p,r_1})$, if $q_1\ge \min\{p,q_0,r_0\}$ and one of the following conditions are satisfied:

     $(i)~p\neq q_0,~~(ii)~~p=q_0\ge r_1,~~(iii)~~r_1<p=q_0<q_1$.
\end{enumerate}
\end{proposition}
Applying Proposition \ref{embl} to $\{2^{sk}\dot{\Delta}_k f\}_{k\in \mathbb Z}$, we have the following embedding result between the homogeneous Besov-Lorentz spaces and Triebel-Lizorkin-Lorentz spaces.
\begin{corollary}\label{embBF}
Let $1<p<\infty$, $1\le q_0,q_1,r_0,r_1<\infty$, $r_0\le r_1$ and $s\in \mathbb R$.
\begin{enumerate}
  \item There holds the embedding $\dot{B}_{(p,r_0),q_0}^s\hookrightarrow \dot{F}_{(p,r_1),q_1}^s$, if $q_0\le \min\{p,q_1,r_1\}$ and one of the following conditions are satisfied:

     $(i)~p\neq q_1,~(ii)~p=q_1\ge r_0,~(iii)~q_0<p=q_1<r_0$.
  \item There holds the embedding $\dot{F}_{(p,r_0),q_0}^s\hookrightarrow \dot{B}_{(p,r_1),q_1}^s$, if $q_1\ge \min\{p,q_0,r_0\}$ and one of the following conditions are satisfied:

     $(i)~p\neq q_0,~(ii)~p=q_0\ge r_1,~(iii)~r_1<p=q_0<q_1$.
\end{enumerate}\end{corollary}
Using Corollaries \ref{equHF} and \ref{embBF}, we can get the embedding result between the homogeneous Besov-Lorentz spaces and Sobolev-Lorentz spaces, which will be omitted. For later use, we only state the embedding result between the homogeneous Besov-Lorentz spaces $\dot{B}_{(p,r_0),2}^s$ and Sobolev-Lorentz spaces $\dot{H}_{p,r_1}^s$.

\begin{corollary}\label{embBH}
Let $1<p<\infty$ and $s\in \mathbb R$.
\begin{enumerate}
  \item There holds the embedding $\dot{B}_{(p,r_0),2}^s\hookrightarrow \dot{H}_{p,r_1}^s$, if one of the following conditions is satisfied:

     ~~~$(i)~p>2~\textnormal{and}~r_1\ge 2,~(ii)~p=2~\textnormal{and}~r_1\ge 2\ge r_0$.

     In particular, we have the embedding $\dot{B}_{(p,2),2}^s\hookrightarrow \dot{H}_{p,2}^s$, if $p\ge2$.
  \item There holds the embedding $\dot{H}_{p,r_1}^s\hookrightarrow \dot{B}_{(p,r_0),2}^s$, if one of the following conditions is satisfied:

     ~~~$(i)~p<2~\textnormal{and}~r_0\le 2,~(ii)~p=2~\textnormal{and}~r_1\ge 2\ge r_0$.

     In particular, we have the embedding $\dot{B}_{(p,2),2}^s\hookrightarrow \dot{H}_{p,2}^s$, if $p\le2$.
\end{enumerate}\end{corollary}

\subsection{Regular Strichartz estimates in Lorentz-type spaces}

In this subsection, we derive the regular Strichartz estimates for $S(t)=e^{\pm it\Delta^{2} } $ in the framework of Lorentz-type spaces.
First, we have the following time decay estimate for  $S(t)=e^{\pm it\Delta^{2} }$ in Besov-Lorentz spaces.
\begin{proposition}\label{decayE}
Let $d\in \mathbb N$, $2\le p<\infty$ and $S(t)=e^{\pm it\Delta^{2} }$. Then we have
$$
\left\|S(t)f\right\|_{\dot{B}_{(p,2),2}^{s(p)}}\lesssim |t|^{-d(1/2-1/p)}\left\|f\right\|_{\dot{B}_{(p',2),2}^{-s(p)}},
$$
where $s(\cdot)=d(1/2-1/\cdot)$.
\end{proposition}
\begin{proof}
The case $p=2$ is trivial, due to the fact that $\dot{B}_{(p,2),2}^{s(p)}=\dot{B}_{(p',2),2}^{-s(p)}=L^2$.
When $p>2$, the result follows from Lemma \ref{intB} and the following regular Strichartz estimates in Besov spaces (see, e.g. \cite[Proposition 3.3]{WHHG11}):
\begin{equation} \nonumber
\left\|S(t)f\right\|_{\dot{B}_{p,2}^{s(p)}}\lesssim |t|^{-d(1/2-1/p)}\left\|f\right\|_{\dot{B}_{p',2}^{-s(p)}},~\forall p\in (2,\infty).
\end{equation}
This completes the proof.
\end{proof}

To derive regular Strichartz estimates, we recall the following well-known result of Keel-Tao \cite{KT98}.
\begin{lemma}[\cite{KT98}]\label{ktS}
Let $\alpha>0$, $H$ be a Hilbert space and $B_0,B_1$ be Banach spaces. Suppose that for each time $t$ we have an operator $U(t):H\to B_0^*$ such that
\begin{equation}\label{GrindEQ__3_7_}
\left\|U(t)\right\|_{H\to B_0^*}\lesssim 1,
\end{equation}
\begin{equation}\label{GrindEQ__3_8_}
\left\|U(t)(U(s))^*\right\|_{B_1\to B_1^*}\lesssim |t-s|^{-\alpha},
\end{equation}
where $(U(s))^*$ denotes the adjoint of $U(s)$ and $B_i^* (i=0,1)$ denotes the dual space of $B_i$.
Let $B_\theta$ denote the real interpolation space $(B_0,B_1)_{\theta,2}$. Then we have the estimates
$$
\left\|U(t)f\right\|_{L^q(\mathbb R, B_{\theta}^*)}\lesssim \left\|f\right\|_{H},
$$
$$
\left\|\int_{s<t}{U(t)(U(s))^*F(s)ds}\right\|_{L^q(\mathbb R, B_{\theta}^*)}\lesssim \left\|F\right\|_{L^{\tilde{q}'}(\mathbb R,B_{\tilde{\theta}})},
$$
whenever $0\le\theta\le 1$, $2\le q=\frac{2}{\alpha \theta}$, $(q,\theta,\alpha)\neq (2,1,1)$, and similarly for $(\tilde{q},\tilde{\theta})$.
\end{lemma}
Using Proposition \ref{decayE} and Lemma \ref{ktS}, we have the following regular Strichartz estimates.
\begin{proposition}[Regular Strichartz estimates in Lorentz-type spaces]\label{regulS}
Let $S(t)=e^{\pm it\Delta^{2} }$ and $s\in \mathbb R$.
Then for any Schr\"{o}dinger admissible pairs $(q,r)\in S$ and $(\tilde{q},\tilde{r})\in S$, we have
\begin{equation} \label{GrindEQ__3_9_}
\left\| S(t)\phi \right\| _{L^{q}(\mathbb R,\dot{B}_{(r,2),2}^{s+2/p})} \lesssim\left\| \phi \right\| _{\dot{H}^{s} },
\end{equation}
\begin{equation} \label{GrindEQ__3_10_}
\left\|\int_{0}^{t}S(t-\tau)f(\tau)d\tau\right\|_{L^{q}(\mathbb R, \dot{B}_{(r,2),2}^{s+2/q})}\lesssim\left\|f\right\|_{L^{\tilde{q}'}(\mathbb R, \dot{B}_{(\tilde{r}',2),2}^{s-2/\tilde{q}})}.
\end{equation}
In \eqref{GrindEQ__3_9_} and \eqref{GrindEQ__3_10_}, replacing homogeneous Besov-Lorentz spaces by corresponding Sobolev-Lorentz spaces, the results still hold.
\end{proposition}
\begin{proof}
Noticing that
$$
(\dot{B}_{p,2}^s)^*=\dot{B}_{p',2}^{-s},~\forall p\in (1,\infty),~\forall s\in \mathbb R,
$$
it follows from Lemma \ref{intB} and the duality theorem (see e.g. \cite[Theorem 3.7.1]{BL76}) that
\begin{equation} \nonumber
(\dot{B}_{(p,2),2}^s)^*=\dot{B}_{(p',2),2}^{-s},~\forall p\in (1,\infty),~\forall s\in \mathbb R.
\end{equation}
Hence, putting $B_0=L^2$ and $B_1=\dot{B}_{(p',2),2}^{-s(p)}$, we have

$$
B_{\tilde{\theta}}=\dot{B}_{(\tilde{r}',2),2}^{-s(\tilde{r})}=\dot{B}_{(p',2),2}^{-2/\tilde{q}},~B_{\theta}^*=\dot{B}_{(r,2),2}^{s(r)}=\dot{B}_{(r,2),2}^{2/q},
$$
where
\footnote[4]{~$p=\infty^{-}$ denotes the fixed number $p>0$ large enough.}
$$
p=\left\{\begin{array}{ll}
{\frac{2d}{d-2}},~&{{\rm if}~d\ge 3,}\\
{\infty^{-},}~&{{\rm if}~d=1,2,}
\end{array}\right.
$$
and
$$
\frac{1}{r}=\frac{1-\theta}{2}+\frac{1}{p},~\frac{1}{\tilde{r}}=\frac{1-\tilde{\theta}}{2}+\frac{1}{\tilde{p}}.
$$
And it follows from Proposition \ref{decayE} that $S(t)$ satisfies \eqref{GrindEQ__3_7_} and \eqref{GrindEQ__3_8_} with
$\alpha=d\left(\frac{1}{2}-\frac{1}{p}\right)$.
We can also see that $2\le q=\frac{2}{\alpha\theta}$ and $2\le \tilde{q}=\frac{2}{\alpha\tilde{\theta}}$, due to the fact that $(q,r)\in S$ and $(\tilde{q},\tilde{r})\in S$.
Hence, \eqref{GrindEQ__3_9_} and \eqref{GrindEQ__3_10_} follow from Lemmas \ref{ktS} and \ref{GrindEQ__3_5_}.
Using Corollary \ref{embBH}, we can see that \eqref{GrindEQ__3_9_} and \eqref{GrindEQ__3_10_} still hold if we replace homogeneous Besov-Lorentz spaces by the corresponding Sobolev-Lorentz spaces. This completes the proof.
\end{proof}
\begin{remark}\label{rem 3.12.}
\textnormal{The regular Strichartz estimates in Besov or Sobolev spaces were initially obtained by \cite{MZ072,P07}. See \cite[Proposition 2]{MZ072} and \cite[Proposition 3.2]{P07}. Proposition \ref{regulS} extends these results to the Lorentz-type spaces.}
\end{remark}
\begin{corollary}\label{cor 3.13.}
Let $S(t)=e^{\pm it\Delta^{2} }$ and $s\in \mathbb R$.
Then, for any $(\bar{q},\bar{r})\in B$ and $(\tilde{q},\tilde{r})\in S$, we have
\begin{equation} \label{GrindEQ__3_11_}
\left\|\int_{0}^{t}S(t-\tau)f(\tau)d\tau\right\|_{L^{\bar{q}}(\mathbb R,\dot{H}_{\bar{r},2}^{s})}\lesssim\left\|f\right\|_{L^{\tilde{q}'}(\mathbb R,\dot{H}_{\tilde{r}',2}^{s-2/\tilde{q}})}.
\end{equation}
\end{corollary}
\begin{proof}
putting $\frac{d+1}{r}:=\frac{d}{2}+\frac{d}{\bar{r}}$, we can see that $(\bar{q},r)\in S$. Furthermore, we can see that $\dot{H}_{r,2}^{s+2/\bar{q}}\hookrightarrow \dot{H}_{\bar{r},2}^s$ by using Lemma \ref{lem 2.4.}. Hence, the result follows from Lemma \ref{regulS}. This completes the proof.
\end{proof}
\section{Proof of Theorem \ref{thm 1.1.}}\label{sec 4.}
In this section, we prove Theorem \ref{thm 1.1.}.
To this end, we establish the following nonlinear estimates.

\begin{lemma}\label{lem 4.1.}
Let $d\ge 3$, $1\le s< \frac{d}{2}$, $0<b<\min\{4,2+\frac{d}{2}-s\}$ and $\sigma=\frac{8-2b}{d-2s}$. If $\sigma$ is not an even integer, assume that $\sigma>\left\lceil s\right\rceil -2$. Then for any interval $I(\subset \mathbb R)$, we have
\begin{equation} \nonumber
\left\||x|^{-b}|u|^{\sigma}u\right\|_{L^{m_0'}(I,\dot{H}_{n_0',2}^{s-1})}
\lesssim \left\|u\right\|^{\sigma+1}_{L^{q_0}(I,\dot{H}_{r_0,2}^{s})},
\end{equation}
where $(m_0,n_0)\in S$ and $(q_0,r_0)\in B_0$ with $n_0=\frac{2d}{d-2}$ and $\frac{1}{r_0}=\frac{1}{2}-\frac{2}{d(\sigma+1)}$.
\end{lemma}
\begin{proof}
An easy computation shows that
\begin{equation} \label{GrindEQ__4_1_}
\frac{1}{n_0'} =\frac{1}{p_1}+\frac{b}{d}=\frac{1}{p_2}+\frac{b+s-1}{d}.
\end{equation}
where
\begin{equation} \label{GrindEQ__4_2_}
\frac{1}{p_1} = \frac{\sigma}{\alpha}+\frac{1}{\beta},~\frac{1}{p_2}=\frac{\sigma+1}{\alpha},~\frac{1}{\alpha}=\frac{1}{r_0} -\frac{s}{d},~\frac{1}{\beta}=\frac{1}{r_0}-\frac{1}{d}.
\end{equation}
We can also verify that
\begin{equation} \nonumber
\frac{1}{r_0}>\frac{s}{d}\Leftrightarrow b<2+\frac{d}{2}-s.
\end{equation}
Hence, using Lemma \ref{lem 2.4.}, it follows from the assumptions $b<2+\frac{d}{2}-s$ and $s\ge 1$ that
\begin{equation} \label{GrindEQ__4_3_}
\dot{H}_{r_0,2}^s\hookrightarrow L^{\alpha,2},~\dot{H}_{r_0,2}^{s}\hookrightarrow \dot{H}_{\beta,2}^{s-1}.
\end{equation}
Using \eqref{GrindEQ__4_1_}--\eqref{GrindEQ__4_3_}, Corollary \ref{cor 2.7.}, Lemmas \ref{lem 2.2.}, \ref{lem 2.5.} and \ref{lem 2.8.}, we have
\begin{eqnarray}\begin{split} \nonumber
\left\| |x|^{-b}|u|^{\sigma}u\right\| _{\dot{H}_{n_0',2}^{s-1} } &\lesssim \left\| |x|^{-b}\right\| _{L^{d/b,\infty}}\left\| |u|^{\sigma}u\right\| _{\dot{H}_{p_1,2}^{s-1}}
+\left\| |x|^{-b}\right\| _{\dot{H}_{d/(b+s-1),\infty}^{s-1}}\left\| |u|^{\sigma}u\right\| _{L^{p_2,2}}\\
&\lesssim \left\| u\right\| _{L^{\alpha,2(\sigma+1)}}^{\sigma}\left\| u\right\| _{\dot{H}_{\beta,2(\sigma+1)}^{s-1}}+\left\| u\right\| _{L^{\alpha,2(\sigma+1)}}^{\sigma+1}\\
&\lesssim \left\| u\right\| _{L^{\alpha,2}}^{\sigma}\left\| u\right\| _{\dot{H}_{\beta,2}^{s-1}}+\left\| u\right\| _{L^{\alpha,2}}^{\sigma+1}
\lesssim \left\| u\right\| _{\dot{H}_{r_0,2}^{s} }^{\sigma +1},
\end{split}\end{eqnarray}
where we used the fact that $\sigma>\left\lceil s\right\rceil -2$ when $\sigma$ is not an even integer.
On the other hand, we can also see that
\begin{equation}\label{GrindEQ__4_4_}
\frac{1}{m_0'}=\frac{\sigma+1}{q_0}.
\end{equation}
Using \eqref{GrindEQ__4_3_}, \eqref{GrindEQ__4_4_} and H\"{o}lder's inequality, we immediately get the desired result.
\end{proof}

\begin{lemma}\label{lem 4.2.}
Let $d\in \mathbb N$, $0\le s< \frac{d}{2}$, $0<b<\min\{4,d\}$ and $\sigma=\frac{8-2b}{d-2s}$. Then for any interval $I(\subset \mathbb R)$, there exist $(m_1,n_1)\in B_0$, $(q_1,r_1)\in B_0$ and $(q_2,r_2)\in B_0$ such that
\begin{equation} \nonumber
\left\||x|^{-b}|u|^{\sigma}v\right\|_{L^{m_1'}(I,L^{n_1',2})}
\lesssim \left\|u\right\|^{\sigma}_{L^{q_1}(I,\dot{H}_{r_1,2}^{s})}\left\|v\right\|_{L^{q_2}(I,L^{r_2,2})}.
\end{equation}
\end{lemma}
\begin{proof}
We first claim that there exist $(m_1,n_1)\in B_0$, $(q_1,r_1)\in B_0$ and $(q_2,r_2)\in B_0$ such that
\begin{equation} \label{GrindEQ__4_5_}
\frac{1}{n_1'} =\sigma \left(\frac{1}{r_1} -\frac{s}{d} \right)+\frac{1}{r_2}+\frac{b}{d}, ~\frac{1}{r_1} -\frac{s}{d} >0.
\end{equation}
In fact, we can see that there exists $(m_1,n_1)\in B_0$ satisfying \eqref{GrindEQ__4_5_} provided that we can prove that there exist $r_1,r_2>0$ such that
\begin{equation} \label{GrindEQ__4_6_}
\left\{\begin{array}{l}
{\max\left\{\frac{d-4}{2d},~\frac{s}{d}\right\}<\frac{1}{r_1}<\frac{1}{2},}\\
{\max\left\{\frac{d-4}{2d},0\right\}<\frac{1}{r_2}<\frac{1}{2},}\\
{\frac{1}{2}<\sigma\left(\frac{1}{r_1}-\frac{s}{d}\right)+\frac{1}{r_2}+\frac{b}{d}<\min\{1,~\frac{1}{2}+\frac{2}{d}\}.}\\
\end{array}\right.
\end{equation}
We divide the study in two cases: $d\ge4$ and $d<4$.
If $d<4$, then we can easily check that there exist $r_2>0$ satisfying \eqref{GrindEQ__4_6_}, if we can choose $r_1>0$ such that
\begin{equation} \label{GrindEQ__4_7_}
\frac{\sigma s}{d}<\frac{\sigma}{r_1}<\min\{\frac{\sigma}{2},1+\frac{\sigma s}{d}-\frac{b}{d}\}.
\end{equation}
It is obvious that there exists $r_1>0$ satisfying \eqref{GrindEQ__4_7_}, due to the fact that $s<\frac{d}{2}$ and $b<d$.
If $d\ge4$, then we can also see that there exists $r_2>0$ satisfying \eqref{GrindEQ__4_6_}, if we can choose $r_1>0$ such that
\begin{equation}\label{GrindEQ__4_8_}
\max\left\{\frac{d-4}{2d},\;\frac{s}{d}\right\}<\frac{1}{r_1}<\min\{\frac{1}{2},\frac{s}{d}+\frac{4-b}{d\sigma}\}.
\end{equation}
We can easily check that there exists $r_1$ satisfying \eqref{GrindEQ__4_8_} by using the fact $s<\frac{d}{2}$, $b<4$ and $\sigma=\frac{8-2b}{d-2s}$.

Using \eqref{GrindEQ__4_5_}, Lemmas \ref{lem 2.1.}, \ref{lem 2.2.} and \ref{lem 2.4.}, we have
\begin{eqnarray}\begin{split} \label{GrindEQ__4_9_}
\left\| |x|^{-b}|u|^{\sigma}v\right\| _{L^{n_1',2}} &\lesssim  \left\| |x|^{-b}\right\| _{L^{d/b,\infty}}\left\| u\right\| _{L^{\bar{\alpha},2}}^{\sigma}\left\| v\right\| _{L^{r_2,2}}
\lesssim \left\| u\right\| _{\dot{H}_{r_1,2}^{s} }^{\sigma }\left\| v\right\| _{L^{r_2,2}},
\end{split}\end{eqnarray}
where $\frac{1}{\bar{\alpha}}:=\frac{1}{r_1}-\frac{s}{d}$.
Meanwhile, \eqref{GrindEQ__4_5_} is equivalent to
\begin{equation}\label{GrindEQ__4_10_}
\frac{1}{m_1'}=\frac{\sigma}{q_1}+\frac{1}{q_2}.
\end{equation}
Using \eqref{GrindEQ__4_9_}, \eqref{GrindEQ__4_10_} and H\"{o}lder's inequality, we immediately get the desired result.
\end{proof}
We are ready to prove Theorem \ref{thm 1.1.}.
\begin{proof}[{\bf Proof of Theorem \ref{thm 1.1.}}]

We first prove the local well-posedness result.
Let $T>0$ and $M>0$ which will be chosen later. Given $I=[-T,T]$, we define the complete metric space $(D_1,d)$ as follows:
\begin{equation} \nonumber
D_1=\left\{u\in \bigcap_{i=0}^{2}L^{q_{i}}(I, H_{r_{i},2}^{s}):~\max_{i\in\{0,1,2\}}\left\|u\right\|_{L^{q_{i}}(I, H_{r_{i}}^{s})} \le M\right\},
\end{equation}
\begin{equation} \nonumber
d(u,v)=\max_{i\in\{0,1,2\}}\left\|u-v\right\|_{L^{q_{i}}(I, L^{r_{i}})},
\end{equation}
where $(q_{i},r_{i})\in B_0$ is given in Lemmas \ref{lem 4.1.} and \ref{lem 4.2.}.
Now we consider the mapping
\begin{equation}\label{GrindEQ__4_11_}
\mathscr{T}:~u(t)\to S(t)u_{0} -i\lambda \int _{0}^{t}S(t-\tau)|x|^{-b}|u(\tau)|^{\sigma}u(\tau)d\tau .
\end{equation}
Lemma \ref{lem 2.9.} (standard Strichartz estimates) and Corollary \ref{cor 3.13.} (regular Strichartz estimates) yield that
\begin{eqnarray}\begin{split} \label{GrindEQ__4_12_}
\max_{i\in\{0,1,2\}}\left\|\mathscr{T}u\right\|_{L^{q_{i}}(I,H_{r_{i},2}^{s})}&\le \max_{i\in\{0,1,2\}}\left\|\mathscr{T}u\right\|_{L^{q_{i}}(I,\dot{H}_{r_{i},2}^{s})}+\max_{i\in\{0,1,2\}}\left\|\mathscr{T}u\right\|_{L^{q_{i}}(I,L^{r_{i},2})}\\
&\le \max_{i\in\{0,1,2\}}\left\|S(t)u_0\right\|_{L^{q_{i}}(I,H_{r_{i},2}^{s})}+C\left\||x|^{-b}|u|^{\sigma}u\right\|_{L^{m_0'}(I,\dot{H}_{n_0',2}^{s-1})}\\
&~~~+C\left\||x|^{-b}|u|^{\sigma}u\right\|_{L^{m_1'}(I,L^{n_1',2})},
\end{split}\end{eqnarray}
where $(m_0,n_0)\in S$ is given in Lemma \ref{lem 4.1.} and $(m_1,n_1)\in B_0$ is as in Lemma \ref{lem 4.2.}.
By Strichartz estimate \eqref{GrindEQ__2_1_}, we can see that
$$
\max_{i\in\{0,1,2\}}\left\|S(t)u_0\right\|_{L^{q_{i}}([-T,T],H_{r_{i},2}^{s})},~\textnormal{as}~T\to 0.
$$
We take $M>0$ satisfying $CM^{\sigma}\le \frac{1}{4}$ and $T>0$ such that
\begin{equation}\label{GrindEQ__4_13_}
\max_{i\in\{0,1,2\}}\left\|S(t)u_0\right\|_{L^{q_{i}}([-T,T],H_{r_{i},2}^{s})}\le \frac{M}{2}.
\end{equation}
Using \eqref{GrindEQ__4_12_}, \eqref{GrindEQ__4_13_}, Lemmas \ref{lem 4.1.} and \ref{lem 4.2.}, we have
\begin{equation}\label{GrindEQ__4_14_}
\max_{i\in\{0,1,2\}}\left\|\mathscr{T}u\right\|_{L^{q_{i}}(I,H_{r_{i}}^{s})} \le \frac{M}{2}+CM^{\sigma+1}\le M.
\end{equation}
Similarly, it follows from Lemmas \ref{lem 2.9.} and \ref{lem 4.2.} that
\begin{eqnarray}\begin{split}\label{GrindEQ__4_15_}
d(\mathscr{T}u,\mathscr{T}v)&\lesssim \left\||x|^{-b}(|u|^{\sigma }+|v|^{\sigma})|u-v|\right\|_{L^{m_1'}(I,L^{n_1',2})} \\
&\le 2C M^{\sigma}d(u,v)\le \frac{1}{2}d(u,v).
\end{split}\end{eqnarray}
Using \eqref{GrindEQ__4_14_}, \eqref{GrindEQ__4_15_} and the standard argument(see e.g. Section 4.3 in \cite{WHHG11}), we can get the desired result.

Let us prove the small data global well-posedness result.
Let $M_1>0$ and $M_2>0$ which will be chosen later. We define the complete metric space $(D_2,d)$ as follows:
\begin{equation} \nonumber
D_2=\left\{u\in \bigcap_{i=0}^{2}L^{q_{i}}(\mathbb R, H_{r_{i},2}^{s}):~\max_{i\in\{0,1,2\}}\left\|u\right\|_{L^{q_{i}}(\mathbb R, H_{r_{i}}^{s})} \le M_1,~\max_{i\in\{0,1,2\}}\left\|u\right\|_{L^{q_{i}}(\mathbb R, \dot{H}_{r_{i}}^{s})} \le M_2\right\},
\end{equation}
\begin{equation} \nonumber
d(u,v)=\max_{i\in\{0,1,2\}}\left\|u-v\right\|_{L^{q_{i}}(\mathbb R, L^{r_{i}})},
\end{equation}
where $(q_{i},r_{i})\in B_0$ is given in Lemmas \ref{lem 4.1.} and Lemma \ref{lem 4.2.}.
Using the similar argument to the above, we have
\begin{equation}\label{GrindEQ__4_16_}
\max_{i\in\{0,1,2\}}\left\|\mathscr{T}u\right\|_{L^{q_{i}}(I,H_{r_{i}}^{s})} \le C \left\| u_{0} \right\| _{H^{s} }+CM_2^{\sigma}M_1,
\end{equation}
\begin{equation}\label{GrindEQ__4_17_}
\max_{i\in\{0,1,2\}}\left\|\mathscr{T}u\right\|_{L^{q_{i}}(I,\dot{H}_{r_{i}}^{s})} \le C \left\| u_{0} \right\| _{\dot{H}^{s} }+CM_2^{\sigma+1},
\end{equation}
\begin{equation}\label{GrindEQ__4_18_}
d(\mathscr{T}u,\mathscr{T}v)\le 2C M_2^{\sigma}d(u,v).
\end{equation}
Putting $M_1=2C\left\| u_{0} \right\| _{H^{s} }$, $M_2=2C\left\| u_{0} \right\| _{\dot{H}^{s} } $ and $\delta=2(4C)^{-\frac{\sigma+1}{\sigma}}$.
If $\left\| u_{0} \right\| _{\dot{H}^{s} }\le \delta$, i.e. $CM_2^{\sigma}<\frac{1}{4}$, it follows from \eqref{GrindEQ__4_16_}--\eqref{GrindEQ__4_18_} that $\mathscr{T}:(D_2,d)\to (D_2,d)$ is a contraction mapping. Hence, there exists a unique global solution of \eqref{GrindEQ__1_1_} in $D_2$. The rest part is proved by using the standard argument (see e.g. \cite{AK23}). This completes the proof.
\end{proof}

\section{Proof of Theorem \ref{thm 1.5.}}\label{sec 5.}
In this section, we prove Theorem \ref{thm 1.5.}.
To this end, we establish the following nonlinear estimates.

\begin{lemma}\label{lem 5.1.}
Let $s>0$, $\sigma>\left\lceil s\right\rceil$ and $1<p,q,r<\infty$ with $\frac{1}{p}=\frac{1}{r}+\frac{\sigma}{q}$.
Assume that $f\in C^{\left\lceil s\right\rceil} \left(\mathbb C\to \mathbb C\right)$ satisfies
\begin{equation} \label{GrindEQ__5_1_}
|f^{(k)} (u)|\lesssim|u|^{\sigma +1-k} ,
\end{equation}
for any $0\le k\le 2$ and $u\in \mathbb C$. Then we have
\begin{eqnarray}\begin{split}\nonumber
\left\| f(u)-f(v)\right\| _{\dot{H}_{p,2}^{s} } \lesssim&(\left\| u\right\| _{L^{q,2}}^{\sigma } +\left\| v\right\| _{L^{q,2}}^{\sigma })\left\| u-v\right\| _{\dot{H}_{r,2}^{s} } \\
&+(\left\| u\right\| _{L^{q,2}}^{\sigma-1} +\left\| v\right\| _{L^{q,2}}^{\sigma-1})
(\left\| u\right\| _{\dot{H}_{r,2}^{s}}+\left\| v\right\| _{\dot{H}_{r,2}^{s}})\left\| u-v\right\| _{L^{q,2}}.
\end{split}\end{eqnarray}
\end{lemma}
\begin{proof}
Without loss of generality and for simplicity, we assume that $f$ is a function of a real variable. We have
\[f(u)-f(v)=\left(u-v\right)\int _{0}^{1}f'\left(v+t\left(u-v\right)\right)dt .\]
Putting
\begin{equation} \label{GrindEQ__5_2_}
\frac{1}{p_{1} } =\frac{\sigma }{q} ,~\frac{1}{p_{2} } =\frac{\sigma-1}{q} +\frac{1}{r},
\end{equation}
it follows Lemmas \ref{lem 2.2.} and \ref{lem 2.5.} that
\begin{equation}\nonumber
\left\| f(u)-f(v)\right\| _{\dot{H}_{p,2}^{s} }=\left\| \left(u-v\right)\int _{0}^{1}f'\left(v+t\left(u-v\right)\right)dt \right\| _{\dot{H}_{p,2}^{s} }\lesssim I_{1} +I_{2},
\end{equation}
where

$$
I_{1}=\left\| \int _{0}^{1}f'\left(v+t\left(u-v\right)\right)dt \right\| _{L^{p_1,2}} \left\| u-v\right\| _{\dot{H}_{r,2}^{s} },
$$

$$
I_{2}=\left\| \int _{0}^{1}f'\left(v+t\left(u-v\right)\right)dt \right\| _{\dot{H}_{p_{2},2 }^{s} } \left\| u-v\right\| _{L^{q,2}}.
$$
First, we estimate $I_{1}$. It follows from \eqref{GrindEQ__5_1_} that
\begin{equation} \label{GrindEQ__5_3_}
|f'\left(v+t\left(u-v\right)\right)|\lesssim{\mathop{\max }\limits_{t\in \left[0,\, 1\right]}} \left|v+t\left(u-v\right)\right|^{\sigma } \lesssim |u|^{\sigma } +|v|^{\sigma },
\end{equation}
for any $0\le t \le 1$. Using \eqref{GrindEQ__5_2_}, \eqref{GrindEQ__5_3_}, Lemmas \ref{lem 2.1.} and \ref{lem 2.2.}, we have
\begin{eqnarray}\begin{split} \label{GrindEQ__5_4_}
I_{1} &\le \left\| u-v\right\| _{\dot{H}_{r,2}^{s} }\int _{0}^{1}\left\| f'\left(v+t\left(u-v\right)\right)\right\| _{L^{p_{1},2} } dt \\
&\lesssim \left\| |u|^{\sigma }+|v|^{\sigma } \right\| _{L^{p_{1},2} }  \left\| u-v\right\| _{\dot{H}_{r}^{s} }\le (\left\| u\right\| _{L^{q,2}}^{\sigma } +\left\| v\right\| _{L^{q,2}}^{\sigma })\left\| u-v\right\| _{\dot{H}_{r,2}^{s} }.
\end{split}\end{eqnarray}
Next, we estimate $I_{2}$. We have
\begin{equation} \label{GrindEQ__5_5_}
\left\| \int _{0}^{1}f'\left(v+t\left(u-v\right)\right)dt \right\| _{\dot{H}_{p_{2},2 }^{s} } \le \int _{0}^{1}\left\| f'\left(v+t\left(u-v\right)\right)\right\| _{\dot{H}_{p_{2},2 }^{s} } dt .
\end{equation}
It also follows from Lemmas \ref{lem 2.2.} and \ref{lem 2.6.} that
\begin{eqnarray}\begin{split} \label{GrindEQ__5_6_}
\left\| f'\left(v+t\left(u-v\right)\right)\right\| _{\dot{H}_{p_{2},2 }^{s} } &\lesssim\left\| f''\left(v+t\left(u-v\right)\right)\right\| _{L^{p_{3},2}} \left\| v+t\left(u-v\right)\right\| _{\dot{H}_{r,2}^{s} }.
\end{split}\end{eqnarray}
where $\frac{1}{p_{3} } =\frac{\sigma -1}{q}$. We can also see that
\begin{equation} \label{GrindEQ__5_7_}
|f''\left(v+t\left(u-v\right)\right)|\lesssim{\mathop{\max }\limits_{t\in \left[0,\, 1\right]}} \left|v+t\left(u-v\right)\right|^{\sigma-1} \lesssim |u|^{\sigma-1} +|v|^{\sigma-1},
\end{equation}
and
\begin{eqnarray} \label{GrindEQ__5_8_}
\begin{split}
\left\| v+t\left(u-v\right)\right\| _{\dot{H}_{r,2}^{s} }& \le {\mathop{\max }\limits_{t\in \left[0,\, 1\right]}} \left\|(-\Delta)^{s/2} v+t[(-\Delta)^{s/2} u-(-\Delta)^{s/2} v]\right\| _{L^{r,2} }\\
 &\lesssim\left\| u\right\| _{\dot{H}_{r,2}^{s} } +\left\| v\right\| _{\dot{H}_{r,2}^{s} },
\end{split}
\end{eqnarray}
for any $0\le t \le 1$.
Using \eqref{GrindEQ__5_2_}, \eqref{GrindEQ__5_5_}--\eqref{GrindEQ__5_8_}, Lemmas \ref{lem 2.1.} and \ref{lem 2.2.}, we have
\begin{eqnarray}\begin{split} \label{GrindEQ__5_9_}
I_{2}&=\left\| \int _{0}^{1}f'\left(v+t\left(u-v\right)\right)dt \right\| _{\dot{H}_{p_{2},2 }^{s} } \left\| u-v\right\| _{L^{q,2}} \\
&\lesssim
(\left\| u\right\| _{L^{q,2}}^{\sigma-1} +\left\| v\right\| _{L^{q,2}}^{\sigma-1})
(\left\| u\right\| _{\dot{H}_{r,2}^{s}}+\left\| v\right\| _{\dot{H}_{r,2}^{s}})\left\| u-v\right\| _{L^{q,2}} .
\end{split}\end{eqnarray}
In view of \eqref{GrindEQ__5_4_} and \eqref{GrindEQ__5_9_}, we get the desired result.
This completes the proof.
\end{proof}

\begin{lemma}\label{lem 5.2.}
Let $d\ge 3$, $1\le s< \frac{d}{2}$, $0<b<\min\{4,2+\frac{d}{2}-s\}$ and $\sigma= \frac{8-2b}{d-2s}$. If $\sigma$ is not an even integer, assume that $\sigma>\left\lceil s\right\rceil -1$. Then for any interval $I(\subset \mathbb R)$, we have
\begin{equation} \nonumber
\left\||x|^{-b}(|u|^{\sigma}u-|v|^{\sigma}v)\right\|_{L^{m_0'}(I,\dot{H}_{n_0',2}^{s-1})}
\lesssim (\left\|u\right\|^{\sigma}_{L^{q_0}(I,\dot{H}_{r_0,2}^{s})}+\left\|u\right\|^{\sigma}_{L^{q_0}(I,\dot{H}_{r_0,2}^{s})})
\left\|u-v\right\|_{L^{q_0}(I,\dot{H}_{r_0,2}^{s})},
\end{equation}
where $(m_0,n_0)\in S$ and $(q_0,r_0)\in B_0$ are as in Lemma \ref{lem 4.1.}.
\end{lemma}
\begin{proof}
The result follows direclty from Lemma \ref{lem 5.1.} and the same argument as in the proof of Lemma \ref{lem 4.1.}.
\end{proof}

We are ready to prove Theorem \ref{thm 1.5.}.
\begin{proof}[{\bf Proof of Theorem \ref{thm 1.5.}}]
For interval $I\subset \mathbb R$, we define the complete metric space $(D_3,d_3)$ as follows:
\begin{equation} \nonumber
D_3=\left\{u\in L^{q_0}(I, \dot{H}_{r_0,2}^{s}):~\left\|u\right\|_{L^{q_0}(I, \dot{H}_{r_0}^{s})} \le M_3\right\},
\end{equation}
\begin{equation} \nonumber
d_3(u,v)=\left\|u-v\right\|_{L^{q_0}(I, \dot{H}_{r_0}^{s})},
\end{equation}
where $(q_0,r_0)\in B_0$ is given in Lemma \ref{lem 4.1.} and $M_3>0$ is determined later.
Considering the mapping $\mathscr{T}$ defined by \eqref{GrindEQ__4_11_}, it follows from  Corollary \ref{cor 3.13.} (regular Strichartz estimates), Lemmas \ref{lem 4.1.} and \ref{lem 5.2.} that
\begin{eqnarray}\begin{split} \label{GrindEQ__5_10_}
\left\|\mathscr{T}u\right\|_{L^{q_0}(I,\dot{H}_{r_0,2}^{s})}
&\le \left\|S(t)u_0\right\|_{L^{q_0}(I,\dot{H}_{r_0,2}^{s})}+C\left\||x|^{-b}|u|^{\sigma}u\right\|_{L^{m_0'}(I,\dot{H}_{n_0',2}^{s-1})}\\
&\le \left\|S(t)u_0\right\|_{L^{q_0}(I,\dot{H}_{r_0,2}^{s})}+C\left\|u\right\|^{\sigma+1}_{L^{q_0}(I,\dot{H}_{r_0,2}^{s})}
\end{split}\end{eqnarray}
and
\begin{eqnarray}\begin{split}\label{GrindEQ__5_11_}
d_3(\mathscr{T}u,\mathscr{T}v)&\lesssim \left\||x|^{-b}(|u|^{\sigma}u-|v|^{\sigma}v)\right\|_{L^{m_0'}(I,\dot{H}_{n_0',2}^{s-1})} \\
&\lesssim(\left\|u\right\|^{\sigma}_{L^{q_0}(I,\dot{H}_{r_0,2}^{s})}+\left\|u\right\|^{\sigma}_{L^{q_0}(I,\dot{H}_{r_0,2}^{s})})
\left\|u-v\right\|_{L^{q_0}(I,\dot{H}_{r_0,2}^{s})},
\end{split}\end{eqnarray}
where $(m_0,n_0)\in S$ is given in Lemma \ref{lem 4.1.}.
Using \eqref{GrindEQ__5_10_}, \eqref{GrindEQ__5_11_} and repeating the argument similar to that in the proof of Theorem \ref{thm 1.1.}, we can get the desired result.
\end{proof}

\emph{E-mail address}: cioc8@ryongnamsan.edu.kp

\emph{E-mail address}: jm.an0221@ryongnamsan.edu.kp

\emph{E-mail address}: jm.kim0211@ryongnamsan.edu.kp


\begin{thebibliography}{35}
\bibitem{AT21} \small{L. Aloui and S. Tayachi, Local well-posedness for the inhomogeneous nonlinear Schr\"{o}dinger equation, \emph{Discrete Contin. Dyn. Syst.}, \textbf{41} (2021), 5409--5437.}
\bibitem{AK23} \small{J. An and J. Kim, The Cauchy problem for the critical inhomogeneous nonlinear Schr\"{o}dinger equation in $H^{s}(\mathbb R^{n} )$, \emph{Evol. Equ. Control Theory}, \textbf{12} (2023), 1039--1055.}
\bibitem{AK232} \small{J. An and J. Kim, A note on the $H^{s}$-critical inhomogeneous nonlinear Schr\"{o}dinger equation, \emph{Z. Anal. Anwend.}, \textbf{42} (2023), 403--433.}
\bibitem{AKR22} \small{J. An, J. Kim and P. Ryu, Local well-posedness for the inhomogeneous biharmonic nonlinear Schr\"{o}dinger equation in Sobolev spaces, \emph{Z. Anal. Anwend.},  {\bf 41} (2022), 239--258.}
\bibitem{AKR24} \small{J. An, J. Kim and P. Ryu, Sobolev-Lorentz spaces with an application to the inhomogeneous biharmonic NLS equation, \emph{Discrete Contin. Dyn. Syst. Ser. B}, \textbf{29} (2024), 3326--3345.}
\bibitem{ARK23} \small{J. An, P. Ryu and J. Kim, Small data global well-posedness for the inhomogeneous biharmonic NLS in Sobolev spaces, \emph{Discrete Contin. Dyn. Syst. Ser. B}, {\bf 28} (2023), 2789--2802.}
\bibitem{BMS23} \small{R. Bai M. Majdoub and T. Saanouni, Non-radial blow-up for a mass-critical fourth-order inhomogeneous nonlinear Schr\"{o}dinger equation, arXiv:2312.07002.}
\bibitem{BS23} \small{R. Bai and T. Saanouni, Finite time blow-up of non-radial solutions for some inhomogeneous Schr\"{o}dinger equations, arXiv:2306.15210.}
\bibitem{BL76} \small{J. Bergh and J. L\"{o}fstr\"{o}m. \textit{Interpolation Spaces. An Introduction.}, Springer, Berlin, 1976.}
\bibitem{CG21} {\small L. Campos and C. M. Guzm\'{a}n, Scattering for the non-radial inhomogenous biharmonic NLS equation, \emph{Calc. Var. Partial Differential Equations}, {\bf 61} (2022), 156.}
\bibitem{CGP22} {\small M. Cardoso, C. M. Guzm\'{a}n and A. Pastor, Global well-posedness and critical norm concentration for inhomogeneous biharmonic NLS, \emph{Monatsh. Math.}, {\bf 198} (2022), 1--29.}
\bibitem{CN16} {\small D. Cruz-Uribe and V. Naibo, Kato-Ponce inequalities on weighted and variable Lebesgue spaces, \emph{Differential Integral Equations.}, {\bf 29} (2016), 801--836.}
\bibitem{D18B} \small{V. D. Dinh, On well-posedness, regularity and ill-posedness of the nonlinear fourth-order Schr\"{o}dinger equation,  \emph{Bull. Belg. Math. Soc. Simon Stevin}, {\bf 25} (2018), 415--437.}
\bibitem{D21} \small{V. D. Dinh, Dynamics of radial solutions for the focusing fourth-order nonlinear Schr\"{o}dinger equations,  \emph{Nonlinearity}, {\bf 34} (2021), 776--821.}
\bibitem{D22} \small{V. D. Dinh, Non-radial finite-time blow-up for the fourth-order nonlinear Schr\"{o}dinger equations, \emph{Appl. Math. Lett.}, \textbf{132} (2022), 108084.}
\bibitem{HS19} \small{P. Hobus and J. Saal, Triebel-Lizorkin-Lorentz spaces and the Navier-Stokes equations, \emph{Z. Anal. Anwend.},  {\bf 38} (2019), 41--72.}
\bibitem{G14} \small{L. Grafakos, \textit{Classical Fourier Analysis}, third ed., Springer, New York, 2014.}
\bibitem{GP20} {\small C. M. Guzm\'{a}n and A. Pastor, On the inhomogeneous biharmonic nonlinear Schr\"{o}dinger equation: Local, global and stability results,  \emph{Nonlinear Anal. Real World Appl.}, {\bf 56} (2020), 103174.}
\bibitem{GP22} {\small C. M. Guzm\'{a}n and A. Pastor, Some remarks on the inhomogeneous biharmonic NLS equation, \emph{Nonlinear Anal. Real World Appl.}, {\bf 67} (2022), 103643.}
\bibitem{HYZ12} \small{H. Hajaiej, X. Yu and Z. Zhai, Fractional Gagliardo-Nirenberg and Hardy inequalities under Lorentz norms, \emph{J. Math. Anal. Appl.}, \textbf{396} (2012), 569--577.}
\bibitem{K96} \small{V. I. Karpman, Stabilization of soliton instabilities by higher-order dispersion: Fourth order nonlinear Schr\"{o}dinger-type equations,  \emph{Phys. Rev. E}, {\bf 53} (1996), 1336--1339.}
\bibitem{KS97} \small{V. I. Karpman and A. G. Shagalov, Solitons and their stability in high dispersive systems. I. Fourth-order nonlinear Schr\"{o}dinger-type equations with power-law nonlinearities, \emph{Phys. Lett. A}, {\bf 228} (1997), 59--65.}
\bibitem{KT98} \small{M. Keel and T. Tao, Endpoint Strichartz estimates, Amer. J. Math., {\bf 120} (1998), 955--980.}
\bibitem{LZ212} \small{X. Liu and T. Zhang, Bilinear Strichartz's type estimates in Besov spaces with application to inhomogeneous nonlinear biharmonic Schr\"{o}dinger equation,  \emph{J. Differential Equations}, {\bf 296} (2021), 335--368.}
\bibitem{LZ21} \small{X. Liu and T. Zhang, The Cauchy problem for the fourth-order Schr\"{o}dinger equation,  \emph{J. Math. Phys.}, {\bf 62} (2021), 071501.}
\bibitem{L23} \small{Z. Lou, Some notes of homogeneous Besov-Lorentz spaces, \emph{J. Math.}, (2023), 5921136.}
\bibitem{MZ072} \small{C. X. Miao and B. Zhang, Global well-posedness of the cauchy problem for nonlinear Schr\"{o}dinger-type equations, \emph{Discrete Contin. Dyn. Syst.}, {\bf 17} (2007), 181--200.}
\bibitem{MAK24} \small{H. Mun, J. An and J. Kim, The Cauchy problem for the fractional nonlinear Schr\"{o}dinger equation in Sobolev spaces, \emph{Bull. Belg. Math. Soc. Simon Stevin}, (2024), to appear.}
\bibitem{O63} \small{R. O'Neil, Convolution operators and $L(p,q)$ spaces, \emph{Duke Math. J.}, \textbf{30} (1963), 129--142.}
\bibitem{P07} \small{B. Pausader, Global well-posedness for energy critical fourth-order Schr\"{o}dinger equations in the radial case, \emph{Dyn. Partial Differ. Equ.,} {\bf 4} (2007), 197--225.}
\bibitem{S21} \small{T. Saanouni, Energy scattering for radial focusing inhomogeneous bi-harmonic Schr\"{o}dinger equations,  \emph{Calc. Var. Partial Differential Equations}, {\bf 60} (2021), 113.}
\bibitem{S22} \small{T. Saanouni, Scattering for radial defocusing inhomogeneous bi-harmonic Schr\"{o}dinger equations,  \emph{Potential Anal.}, {\bf 56} (2022), 649--667.}
\bibitem{SG22} \small{T. Saanouni and R. Ghanmi, A note on the inhomogeneous fourth-order Schr\"{o}dinger equation,  \emph{J. Pseudo-Differ. Oper. Appl.}, {\bf 56} (2022), 13:56.}
\bibitem{SP23} \small{T. Saanouni and C. Peng, Local well-posedness of a critical inhomogeneous bi-harmonic Schr\"{o}dinger equation, \emph{Mediterr. J. Math.}, {\bf 20} (2023), 170.}
\bibitem{SG24} \small{T. Saanouni and R. Ghanmi, Local well-posedness of a critical inhomogeneous bi-harmonic Schr\"{o}dinger equation, \emph{Advances in Opertor Theory.}, {\bf 9} (2024), (2662-2009).}
\bibitem{ST19} \small{A. Seeger and W. Trebels, Embeddings for spaces of Lorentz-Sobolev type, \emph{Math. Ann.}, {\bf 373} (2019), 1017--1056.}
\bibitem{T78} \small{H. Triebel, \textit{Interpolation Theory, Function Spaces, Differential Operators}, Deutscher Verlag der Wissenschaften, Berlin, 1978.}
\bibitem{T83} \small{H. Triebel. \textit{Theory of function spaces.} Monographs in Mathematics, 78. Birkh\"{a}user Verlag, Basel, 1983.}
\bibitem{WHHG11} \small{B. X. Wang, Z. Huo, C. Hao and Z. Guo, \textit{Harmonic Analysis Method for Nonlinear Evolution Equations, I}, World Scientific Publishing, Hackensack, 2011.}
\bibitem{XY11} \small{Z. Xiang and W. Yan. On the well-posedness of the quasi-geostrophic equation in the Triebel-Lizorkin-Lorentz spaces. \emph{J. Evol. Equ.}, {\bf 11} (2011), 241--263.}
\bibitem{XY12} \small{Z. Xiang and W. Yan. On the well-posedness of the Boussinesq equation in the Triebel-Lizorkin-Lorentz spaces. \emph{Abstr. Appl. Anal.},  (2012), 573087.}
\bibitem{YCP05} \small{Q. Yang, Z. Cheng and L. Peng, Uniform characterization of function spaces by wavelets, \emph{Acta Math. Sci. Ser. A Chin. Ed}, {\bf 25} (2005), 130--144.}
\bibitem{ZWT11} \small{X. Zhong, X. P. Wu and C. L. Tang, Local well-posedness for the homogeneous Euler equations,  \emph{Nonlinear Anal.},  {\bf 74} (2011), 3829--3848.}
\end{thebibliography}
\end{document}